\documentclass[12pt]{article}

\usepackage[latin1]{inputenc}
\usepackage[english]{babel}
\usepackage{amssymb,amsmath}
\usepackage{amsthm}
\usepackage{amsfonts}
\usepackage{graphicx}
\usepackage[margin=3.0cm]{geometry}
\usepackage{enumerate}
\usepackage{color}
\usepackage[noadjust]{cite}
\usepackage{tikz}
\usepackage{mathtools}
\usepackage{framed}
\usepackage{subfig}
\usepackage{hyperref}

\newtheorem{thm}{Theorem}

\newtheorem{cor}[thm]{Corollary}

\newtheorem{lem}[thm]{Lemma}
\newtheorem{prop}[thm]{Proposition}

\theoremstyle{definition}

\newtheorem*{conj*}{Conjecture}

\newcommand{\mymod}{{\textrm{mod} \ }}

{
  \begin{framed}\sffamily \textbf{#1}%
}%
{\rmfamily\end{framed}}

\DeclarePairedDelimiter{\floor}{\lfloor}{\rfloor}

\newcommand{\Path}[1]{\ensuremath{{\sf P}_{#1}}}
\newcommand{\Q}[1]{\ensuremath{{\sf Q}_{#1}}}
\newcommand{\U}[1]{\ensuremath{{\sf U}_{#1}}}

\newcommand{\C}[1]{\ensuremath{{\sf C}_{#1}}}
\newcommand{\CBar}[1]{\ensuremath{{\sf \overline{C}}_{#1}}}

\newcommand{\K}[1]{\ensuremath{{\sf K}_{#1}}}
\newcommand{\B}[1]{\ensuremath{{\sf B}_{#1}}}
\newcommand{\TBell}[1]{\ensuremath{{\sf T}_{#1}}}
\newcommand{\ABell}[1]{\ensuremath{{\sf A}_{#1}}}
\newcommand{\E}[1]{\ensuremath{\overline{{\sf K}}_{#1}}}

\newcommand{\PHOEG}{\emph{PHOEG}}
\newcommand{\ncol}{\ensuremath{\mathcal{B}}}
\newcommand{\avcol}{\ensuremath{\mathcal{A}}}
\newcommand{\totcol}{\ensuremath{\mathcal{T}}}
\newcommand{\ncolk}{\ensuremath{S}}

\newcommand{\chrom}{\ensuremath{\chi}}

\newcommand{\T}{\totcol}
\newcommand{\A}{\avcol}

\newcommand{\merge}[2]{\ensuremath{#1_{\mid #2}}}

\usetikzlibrary{shapes, graphs, graphs.standard, calc}
\tikzstyle{vertex}=[circle, draw, fill=black, minimum size=8pt, inner sep=0pt]

\makeatletter  

\title{Upper bounds on the average number of colors in the non-equivalent colorings of a graph}

\author{
		Alain Hertz\textsuperscript{1},
		Hadrien M\'elot\textsuperscript{2},
		S\'ebastien Bonte\textsuperscript{2},\\
		Gauvain Devillez\textsuperscript{2},
		Pierre Hauweele\textsuperscript{2},
		\\[3mm]
		\footnotesize \textsuperscript{1} Department of Mathematics and Industrial
	Engineering\\
	\footnotesize Polytechnique Montr\'eal - Gerad, Montr\'eal, Canada\\
	\footnotesize Corresponding author. Email: alain.hertz@gerad.ca\\[3mm]
	 \footnotesize \textsuperscript{2} Computer Science Department - Algorithms Lab\\
	 \footnotesize University of Mons, Mons, Belgium
	}

\begin{document}

\maketitle
\vspace*{0.2cm}

\hrule
\vspace*{0.2cm}
\small
\noindent
\textbf{Abstract.} \\
A coloring of a graph is an assignment of colors to its vertices such that adjacent vertices have different colors. Two colorings are equivalent if they induce the same partition of the vertex set into color classes.
Let $\avcol(G)$ be the average number of colors in the non-equivalent colorings of a graph $G$. We give a general upper bound on $\avcol(G)$ that is valid for all graphs $G$ and a more precise one for graphs $G$ of order $n$ and maximum degree $\Delta(G)\in \{1,2,n-2\}$.

\vspace*{0.2cm}
\noindent
\emph{Keywords:} graph coloring, average number of colors, graphical Bell numbers.

\vspace*{0.2cm}
\hrule

\normalsize

\section{Introduction} \label{sec_intro}

A coloring of a graph $G$ is an assignment of colors to its vertices such that adjacent vertices have different colors. 
The total number $\ncol(G)$ of non-equivalent colorings (i.e., with different partitions into color classes) of a graph $G$ is the number of partitions of the vertex set of $G$ whose blocks are stable
sets (i.e., sets of pairwise non-adjacent vertices). This invariant has been studied by several authors in the last few years \cite{absil,Duncan10,DP09,GT13,Hertz16,KN14} under the name of (graphical) Bell number. 
It is related to the standard Bell number $\B{n}$ (sequence A000110 in OEIS \cite{Sloane}) that corresponds to the number of partitions of a set of $n$ elements into non-empty subsets, and is thus obviously the same as the number of non-equivalent colorings of the empty graph or order $n$ (i.e., the graph with $n$ vertices and without any edge).

The 2-Bell number $\TBell{n}$ (sequence A005493 in OEIS \cite{Sloane}) is the total number of blocks in all partitions of a set of $n$ elements. Odlyzko and Richmond \cite{OR} have studied the average number $\ABell{n}$ of blocks in a partition of a set of $n$ elements, which can be defined as $\ABell{n}=\frac{\TBell{n}}{\B{n}}.$ The corresponding concept in graph theory is the average number $\avcol(G)$ of colors in the non-equivalent  colorings of a graph $G$. This graph invariant was recently defined in \cite{Hertz21}. When constraints (represented by edges in $G$) impose that certain pairs of elements (represented by vertices) cannot belong to the same block of a partition, $\avcol(G)$ is the average number of blocks in the partitions that respect all constraints. Clearly, $\avcol(G)=\ABell{n}$ if $G$ is the empty graph of order $n$. 

Lower bounds on $\avcol(G)$ are studied in \cite{HertzLB}. The authors mention that there is no known lower bound on $\avcol(G)$ which is a function of $n$ and such that there exists at least one graph of order $n$ which reaches it. As we will show, the situation is not the same for the upper bound. Indeed, we show that there is an upper bound on $\avcol(G)$ which is a function of $n$ and such that there exists at least one graph of order $n$ which reaches it. We also give a sharper upper bound for graphs  with maximum degree  $\Delta(G)\in \{1,2,n-2\}$.

In the next section we fix some notations. Section \ref{sec_prop} is devoted to properties of $\avcol(G)$ and basic ingredients that we will use in Section \ref{sec_ub} for proving the validity of the upper bounds on $\avcol(G)$.

\section{Notation} \label{sec_nota}

 \allowdisplaybreaks
 For basic notions of graph theory that are not defined here, we refer to Diestel~\cite{Diestel00}. The order  of a graph $G=(V,E)$ is its number $|V|$ of vertices, and the size of $G$ is its number $|E|$ of edges. We write $\overline{G}$ for the complement of $G$ and $G \simeq H$ if $G$ and $H$ are two isomorphic graphs. We denote by $\K{n}$ (resp. $\C{n}$, $\Path{n}$ and $\E{n}$) the \emph{complete graph} (resp. the \emph{cycle}, the \emph{path} and the empty graph) of order $n$. For a subset $W$ of vertices in $G$, we write $G[W]$ for the subgraph induced by $W$. Given two graphs $G_1$ and $G_2$ (with disjoint sets of vertices), we write $G_1\cup G_2$ for the \emph{disjoint union} of $G_1$ and $G_2$. 
  Also, $G\cup p\K{1}$ is the graph obtained from $G$ by adding $p$ isolated vertices, i.e. $G\cup p\K{1}\simeq G\cup \E{p}$.

Let $N(v)$ be the set of vertices adjacent to a vertex $v$ in $G$. We say that $v$ is \emph{isolated} if $|N(v)| = 0$. We write $\Delta(G)$ for the \emph{maximum degree} of $G$. A vertex $v$ of a graph $G$ is \emph{simplicial} if the induced subgraph $G[N(v)]$ of $G$ is a clique.

Let $u$ and $v$ be any two vertices in a graph $G$ of order $n$. We use the following notations:
\begin{itemize}\setlength\itemsep{-3pt}
	\item $\merge{G}{uv}$ is the graph (of order $n-1$) obtained by identifying (merging) the vertices $u$ and $v$ and, if $uv \in E(G)$, by removing the edge $uv$;
	\item if $uv \in E(G)$, $G - uv$ is the graph obtained by removing the edge $uv$ from $G$;
	\item if $uv \notin E(G)$, $G + uv$ is the graph obtained by adding the edge $uv$ in $G$;
	\item $G - v$ is the graph obtained from $G$ by removing $v$ and all its incident edges.
\end{itemize}

A \emph{coloring} of a graph $G$ is an assignment of colors to the vertices of $G$ such that adjacent vertices have different colors. The \emph{chromatic number} \chrom(G) is the minimum number of colors in a coloring of $G$. Two colorings are \emph{equivalent} if they induce the same partition of the vertex set into color classes. Let $\ncolk(G,k)$ be the number of non-equivalent colorings of a graph $G$ that use \emph{exactly} $k$ colors. Then, the total number $\ncol(G)$ of non-equivalent colorings of a graph $G$ is defined by 
$$\ncol(G) = \sum_{k = \chrom(G)}^n \ncolk(G, k),$$
and the total number $\totcol(G)$ of color classes in the non-equivalent colorings of a graph $G$ is defined by
$$\totcol(G) = \sum_{k = \chrom(G)}^n k \ncolk(G, k).$$
In this paper, we study the average number $\A(G)$ of colors in the non-equivalent colorings of a graph $G$, that is,
$$\avcol(G) =\frac{\totcol(G)}{\ncol(G)}.$$
For illustration, as shown in Figure \ref{fig:P4}, there are one non-equivalent coloring of $\Path{4}$ with 2 colors, three with 3 colors, and one with 4 colors, which gives $\ncol(\Path{4})=5$, $\totcol(\Path{4})=15$ and $\A(\Path{4})=\frac{15}{5}=3.$

\begin{figure}[!hbtp]
	\centering
	\includegraphics[scale = 1.1]{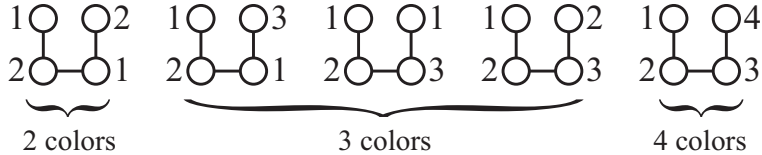}
	\caption{The non-equivalent colorings of $\Path{4}$.}\label{fig:P4}
\end{figure}

\section{Properties of $\ncolk(G,k)$ and $\A(G)$} \label{sec_prop}

As for several other invariants in graph coloring, the \emph{deletion-contraction} rule (also often called the {\it Fundamental Reduction Theorem}~\cite{DKT05}) can be used to compute $\ncol(G)$ and $\T(G)$.  More precisely, let $u$ and $v$ be any pair of distinct vertices of $G$. As shown in~\cite{DP09,KN14}, we have 
\begin{align}
S(G, k) = S(G - uv, k) - S(G_{\mid uv}, k)& \quad \forall uv \in E(G),\label{recs_minus}\\
S(G, k) = S(G + uv, k) + S(G_{\mid uv}, k) & \quad \forall uv \notin E(G)\label{recs_plus}.
\end{align}
It follows that
\begin{align}
\left.
\begin{array}{ll}
\ncol(G) = \ncol(G - uv) - \ncol(\merge{G}{uv})\\
\T(G) = \T(G - uv) - \T(\merge{G}{uv})
\end{array}
\right\}& \quad \forall uv \in E(G),\label{rec_minus}\\
\left.
\begin{array}{ll}
\ncol(G) = \ncol(G + uv) + \ncol(\merge{G}{uv})\\
\T(G) = \T(G + uv) + \T(\merge{G}{uv})
\end{array}
\right\}& \quad \forall uv \notin E(G).\label{rec_plus}
\end{align}

Many properties on $\avcol(G)$ are proved in \cite{Hertz21} and \cite{HertzLB}. We mention here some of them that will be useful for proving the validity of the upper bounds on $\avcol(G)$ given in Section \ref{sec_ub}.

\begin{prop}[\cite{HertzLB}] \label{prop:removeSimplicialEdge}
Let $v$ be a simplicial vertex of degree at least one in a graph $G$, and let $w$ be one of its neighbors in $G$.
Then $\avcol(G)>\avcol(G-vw)$.
\end{prop}

\begin{prop}[\cite{HertzLB}] \label{prop:crossProductAvCol}
Let $H_1$ and $H_2$ be any two graphs. If $\ncolk(H_1,k) \ncolk(H_2,k') \geq \ncolk(H_2,k)\ncolk(H_1,k')$ for all $k {>} k'$, the inequality being strict for at least one pair $(k,k')$, then ${\avcol(G\cup H_1) > \avcol(G\cup H_2)}$ for all graphs $G$.
\end{prop}

\begin{prop}[\cite{Hertz21}] \label{prop:pro_avcolInfSum}  Let $G, H$ and $F_1,\cdots,F_r$ be $r+2$ graphs, and let $\alpha_1,\cdots,\alpha_r$ be $r$ positive numbers such that
	\begin{itemize}\setlength\itemsep{0.1em}
		\item $\ncol(G)=\ncol(H)+\displaystyle\sum_{i=1}^r\alpha_i\ncol(F_i)$
		\item $\T(G)=\T(H)+\displaystyle\sum_{i=1}^r\alpha_i\T(F_i)$
		\item $\A(F_i)<\A(H)$ for all $i=1,\cdots,r$.
	\end{itemize}
	Then $\A(G)<\A(H)$.
\end{prop}

\begin{prop}[\cite{HertzLB}]\label{lem:12}
	$\A(G\cup \C{n})>\A(G\cup \Path{n})$ for all $n\geq 3$ and all graphs $G$.
\end{prop}

Some graphs $G$ of order $n\leq 9$ will play a special role in the next section. The values $\ncolk(G,k)$ of these graphs, with $2\leq k\leq n$, are given in Table \ref{tab_cross}. These values lead to the following lemma.

\renewcommand{\arraystretch}{1.0}
\begin{table}[!htb]
	\caption{Values of $S(G, k)$ for some graphs $G$ of order $n$ and $2\leq k\leq n$ }\label{tab_cross}
	\begin{center}
		\begin{tabular}{ l| c c c c c c c c c }
			\multicolumn{1}{c}{$k$} & 2 & 3 & 4 & 5 & 6 & 7 & 8 & 9 & 10 \\ \hline
			$\ncolk(\C{3} \cup \K{2},k)$ & 0 & 6 & 6 & 1  & && & & \\
			$\ncolk(\C{4} \cup \K{1},k)$ & 2 & 7 & 6 & 1 & & & & & \\
			$\ncolk(\C{5},k)$ & 0 & 5 & 5 & 1  & & & & & \\ 
			$\ncolk(2\C{3},k)$ & 0 & 6 & 18 & 9 & 1 & & & &\\
			$\ncolk(\C{4} \cup \K{2},k)$ & 2 & 16 & 25 & 10 & 1& & & &\\ 
			$\ncolk(\C{5} \cup \K{1},k)$ & 0 & 15 & 25 & 10 & 1& & & & \\
			$\ncolk(\C{6},k)$ & 1 & 10 & 20 & 9 & 1 & & & &\\ 
			$\ncolk(\C{3} \cup \C{4},k)$ & 0 & 18 & 66 & 55 & 14 & 1 & & & \\
			$\ncolk(\C{5} \cup \K{2},k)$ & 0 & 30 & 90 & 65 & 15 & 1& &&\\
			$\ncolk(\C{7},k)$ & 0 & 21 & 70 & 56 & 14 & 1 & & &\\ 
			$\ncolk(\C{3} \cup \C{5},k)$  & 0 & 30 & 210 & 285 & 125 & 20 & 1 & & \\
			$\ncolk(2\C{4},k)$   & 2 & 52 & 241 & 296 & 126 & 20 & 1 & & \\
			$\ncolk(\C{8},k)$   & 1 & 42 & 231 & 294 & 126 & 20 & 1 & &\\ 
			$\ncolk(3 \C{3},k)$ & 0 & 36 & 540 & 1242 & 882 & 243 & 27 & 1 &\\
			$\ncolk(\C{3} \cup \C{6},k)$ & 0 & 66 & 666 & 1351 & 910 & 245 & 27 & 1 & \\
			$\ncolk(\C{4} \cup \C{5},k)$ & 0 & 90 & 750 & 1415 & 925 & 246 & 27 & 1 &\\
			$\ncolk(\C{9},k)$ & 0 & 85 & 735 & 1407 & 924 & 246 & 27 & 1 & \\ 
			$\ncolk(2 \C{3} \cup \C{4},k)$ & 0 & 108 & 1908 & 5838 & 5790 & 2361 & 433 & 35 & 1\\
			$\ncolk(2 \C{5} ,k)$ & 0 & 150 & 2250 & 6345 & 6025 & 2400 & 435 & 35 & 1\\ \hline
		\end{tabular}
	\end{center}
\end{table}

\begin{lem}\label{lem:11inequalities}
	The following strict inequalities are valid for all graphs $G$:
	\begin{center}
		\addtolength{\tabcolsep}{0pt}		
		\begin{tabular}{ c l ccl}
			\emph{(a)} & $\A(G \cup \C{6})  < \A(G \cup 2\C{3})$&$\quad$&
			\emph{(b)} & $\A(G \cup \C{7})  < \A(G \cup \C{3} \cup \C{4}) $\\
			\emph{(c)} & $\A(G \cup \C{8})  < \A(G \cup \C{3} \cup \C{5}) $&&
			\emph{(d)} & $\A(G \cup \C{3} \cup \K{2})  <  \A(G \cup \C{5})$\\
			\emph{(e)} & $\A(G \cup \C{4} \cup \K{2})  <  \A(G \cup 2\C{3})$&&
			\emph{(f)} & $\A(G \cup \C{5} \cup \K{2})  <  \A(G \cup \C{3} \cup \C{4})$\\
			\emph{(g)} & $\A(G \cup \C{4} \cup \K{1})  <  \A(G \cup \C{5})$&&
			\emph{(h)} & $\A(G \cup \C{5} \cup \K{1})  <  \A(G \cup 2 \C{3})$\\
			\emph{(i)} & $\A(G \cup 2\C{4})  <  \A(G \cup \C{3} \cup \C{5})$&&
			\emph{(j)} & $\A(G \cup \C{4} \cup \C{5})  <  \A(G \cup 3\C{3})$\\
			\emph{(k)} & $\A(G \cup 2\C{5})  <  \A(G \cup 2\C{3} \cup \C{4}).$&&&
		\end{tabular}
	\end{center}
\end{lem}
\begin{proof}
	All these inequalities can be obtained from Proposition~\ref{prop:crossProductAvCol} by using the values given in Table~\ref{tab_cross}. For example, to check that (a) holds, the $4^{th}$ and $7^{th}$ lines of Table~\ref{tab_cross} allow to check that $ \ncolk(2\C{3},k)\ncolk(\C{6},k')-\ncolk(\C{6},k) \ncolk(2\C{3},k') \geq 0$ for all $k > k'$ and at least one of these values is strictly positive. 
\end{proof}


We now show the validity of four lemmas which will be helpful for proving that
$\avcol(G \cup \C{n}) < \avcol(G \cup \C{n-3} \cup \C{3})$ for all $n\geq 6$. A direct consequence of this result will be that a graph $G$ that maximizes $\avcol(G)$ among the graphs with maximum degree $2$ cannot contain an induced $\C{n}$ with $n\geq6$.
\begin{lem} \label{prop:ncolkCn}
	$\ncolk(\C{n},k) = (k-1)\ncolk(\C{n-1},k)+\ncolk(\C{n-1},k-1)$ for all $n \geq 4$ and all $k \geq 3$.
\end{lem}

\begin{proof}
	The values in the following table show that the result is true for $n=4$.	
	
	\begin{center}
		\begin{tabular}{ c| c c c }
			k & 2 & 3 & 4 \\ \hline
			$\ncolk(\C{4},k)$ & 1 & 2 & 1 \\
			$\ncolk(\C{3},k)$ & 0 & 1 & 0 
		\end{tabular}
	\end{center}For larger values of $n$, we proceed by induction. So assume $n\geq 5$, let $u$ be a vertex in $\C{n}$, and let $v$ and $w$ be its two neighbors in $\C{n}$. Let us analyze the set of non-equivalent colorings of $\C{n}$ that use exactly $k$ colors:
	\begin{itemize}\setlength\itemsep{-0.0em}
		\item there are $(k-1)\ncolk(\C{n-2},k)$ such colorings where $v$ and $w$ have the same color and at least one vertex of $\C{n}-u$ has the same color as $u$;
		\item there are $\ncolk(\C{n-2},k-1)$ such colorings where $v$ and $w$ have the same color and no vertex on $\C{n}-u$ has the same color as $u$;
		\item there are $(k-2)\ncolk(\C{n-1},k)$ such colorings where $v$ and $w$ have different colors and at least one vertex of $\C{n}-u$ has the same color as $u$;
		\item there are $\ncolk(\C{n-1},k-1)$ such colorings where $v$ and $w$ have different colors and no vertex on $\C{n}-u$ has the same color as $u$.
	\end{itemize}
	Hence, 
	\begin{align*}
	\ncolk(\C{n},k) &= \Big((k-1)\ncolk(\C{n-2},k)+\ncolk(\C{n-2},k-1)\Big) + (k-2)\ncolk(\C{n-1},k)+\ncolk(\C{n-1},k-1)  \\
	&= \ncolk(\C{n-1},k)+(k-2)\ncolk(\C{n-1},k){+}\ncolk(\C{n-1},k-1) 
	\\&= (k-1)\ncolk(\C{n-1},k){+}\ncolk(\C{n-1},k-1). 
	\end{align*}
\end{proof}

\begin{lem} \label{lem:ncolkCUnion2}
	If $n\geq 7$ and $k\leq n$ then $$\ncolk(\C{n-3} \cup \C{3}, k) = (k-1)\ncolk(\C{n-4}\cup \C{3},k)+\ncolk(\C{n-4} \cup \C{3},k-1)-(-1)^n\delta_k$$
	where $$\delta_k = \begin{cases}
	6& \text{ if } k=3,4,\\
	1& \text{ if } k=5,\\
	0& \text{ otherwise } .
	\end{cases}$$	
\end{lem}

\begin{proof}
	The values in the following table show that the result is true for $n=7$.	
	
	\begin{center}
		\begin{tabular}{ c| c c c c c c}
			k & 2 & 3 & 4 & 5 & 6 & 7\\ \hline
			$\ncolk(\C{4} \cup \C{3},k)$ & 0 & 18 & 66 & 55 & 14 & 1 \\
			$\ncolk(\C{3} \cup \C{3},k)$ & 0 & 6 & 18 & 9 & 1 & 0
		\end{tabular}
	\end{center}
	For larger values of $n$, we proceed by induction. Let $u$ be a vertex in $\C{n-3}$, and let $v$ and $w$ be its two neighbors. We analyze the set of non-equivalent colorings of $\C{n-3} \cup \C{3}$ that use exactly $k$ colors:
	\begin{itemize}\setlength\itemsep{-0.0em}
		\item there are $(k-1)\ncolk(\C{n-5} \cup \C{3},k)$ such colorings where $v$ and $w$ have the same color and at least one vertex of $\C{n-3} \cup \C{3}-u$ has the same color as $u$;
		\item there are $\ncolk(\C{n-5} \cup \C{3},k-1)$ such colorings where $v$ and $w$ have the same color and no vertex on $\C{n-3} \cup \C{3}-u$ has the same color as $u$;
		\item there are $(k-2)\ncolk(\C{n-4} \cup \C{3},k)$ such colorings where $v$ and $w$ have different colors and at least one vertex of $\C{n-3} \cup \C{3}-u$ has the same color as $u$;
		\item there are $\ncolk(\C{n-4} \cup \C{3},k-1)$ such colorings where $v$ and $w$ have different colors and no vertex on $\C{n-3} \cup \C{3}-u$ has the same color as $u$.
	\end{itemize}
	Hence, 
	\begin{align*}	
	\ncolk(\C{n-3} \cup \C{3}, k) =& \Big((k-1)\ncolk(\C{n-5} \cup \C{3},k) + \ncolk(\C{n-5} \cup \C{3},k-1)\Big)\\&+ (k-2)\ncolk(\C{n-4} \cup \C{3},k) + \ncolk(\C{n-4} \cup \C{3},k-1)\\
	=& \Big(\ncolk(\C{n-4} \cup \C{3},k) {+} (-1)^{n-1}\delta_k\Big) \\&{+} (k-2)\ncolk(\C{n-4} \cup \C{3},k) {+} \ncolk(\C{n-4} \cup \C{3},k-1) \\
=& (k-1)\ncolk(\C{n-4}\cup \C{3},k)+\ncolk(\C{n-4} \cup \C{3},k-1)-(-1)^n\delta_k.
	\end{align*}
\end{proof}

For $n\geq 3$, let $\Q{n}$ be the graph obtained from $\Path{n}$ by adding an edge between an extremity $v$ of $\Path{n}$ and the vertex at distance 2 from $v$ on $\Path{n}$.

\begin{lem} \label{lem:ncolkCUnion}
	If $n\geq 6$ and $k\leq n$ then $\ncolk(\C{n-3} \cup \C{3}, k) = \ncolk(\Q{n},k) - (-1)^n \rho_k$
	where $$\rho_k = \begin{cases}
	2& \text{ if } k=3,\\
	1& \text{ if } k=4,\\
	0& \text{ otherwise } .
	\end{cases}$$
\end{lem}

\begin{proof}
	
	The values in the following table show that the result is true for $n=6$.
	
	\begin{center}
		\begin{tabular}{ c| c c c c c c c}
			k & 2 & 3 & 4 & 5 & 6\\ \hline
			$\ncolk(2\C{3},k)$ & 0 & 6 & 18 & 9 & 1\\ 
			$\ncolk(\Q{6},k)$ & 0 & 8 & 19 & 9 & 1 \\
		\end{tabular}
	\end{center}
	
	For larger values of $n$, we proceed by induction. Equations \eqref{recs_minus} and \eqref{recs_plus} give
	\begin{align*}
	\ncolk(\C{n-3} \cup \C{3}, k) &= \ncolk(\Path{n-3} \cup \C{3}, k) - \ncolk(\C{n-4} \cup \C{3},k) \\
	&=\ncolk(\Path{n-3} \cup \Path{3}, k)-\ncolk(\Path{n-3} \cup \Path{2}, k)- \ncolk(\C{n-4} \cup \C{3},k)\\
	&=\ncolk(\Path{n},k)+\ncolk(\Path{n-1},k)-
	\ncolk(\Path{n-1},k)-\ncolk(\Path{n-2},k)- \ncolk(\C{n-4} \cup \C{3},k)\\
	&=\ncolk(\Q{n},k)+\ncolk(\Q{n-1},k)- \ncolk(\C{n-4} \cup \C{3},k)\\
	&=\ncolk(\Q{n},k)+(-1)^{n-1} \rho_k\\
	&=\ncolk(\Q{n},k)-(-1)^{n} \rho_k
	\end{align*}
	
\end{proof}

\begin{lem} \label{prop:ncolkCn2}The following inequalities are valid for all $n\geq 9$:
	\begin{itemize}\setlength\itemsep{-0.1em}
		\item[\emph{(a)}]$\ncolk(\C{n},k)> \ncolk(\C{n},k-1)$ for all $k\in \{3,4,5\}$;
		\item[\emph{(b)}]$\ncolk(\C{n},k) > 3 \ncolk(\C{n-1},k-1)$ for all $k \in \{3,4,5,6\}$;
		\item[\emph{(c)}]$\ncolk(\C{n},4) > 8 \ncolk(\C{n},3)$.
	\end{itemize}
\end{lem}

\begin{proof}
	The values in Table \ref{tab_cross} show that the inequalities are satisfied for $n=9$.
	For larger values of $n$, we proceed by induction. 
	Note that (a) and (b) are clearly valid for $k=3$ since $\ncolk(\C{n},3)> 3 \geq \max\{\ncolk(\C{n},2)),3\ncolk(\C{n-1},2)\}$. We may therefore assume $k\in\{4,5\}$ for (a) and $k\in\{4,5,6\}$ for (b). Lemma \ref{prop:ncolkCn} and the induction hypothesis imply 
	\begin{align*}\ncolk(\C{n},k)=&(k-1)\ncolk(\C{n-1},k)+\ncolk(\C{n-1},k-1)
	\\> &(k-2)\ncolk(\C{n-1},k-1)+\ncolk(\C{n-1},k-2)\\
	=&\ncolk(\C{n},k-1).
	\end{align*}
	Hence (a) is proved. It follows that the following inequality is valid:
	\begin{align*}
	\frac{1}{k-1}\ncolk(\C{n-1},k-1)=&\frac{1}{k-1}\Big((k-2)\ncolk(\C{n-2},k-1)+\ncolk(\C{n-2},k-2)\Big)\\
	<& \frac{1}{k-1}\Big((k-1)\ncolk(\C{n-2},k-1)\Big)\\
	=&\ncolk(\C{n-2},k-1)
	\end{align*} 
	which implies
	\begin{align*}\ncolk(\C{n},k)=&(k-1)\ncolk(\C{n-1},k)+\ncolk(\C{n-1},k-1)\\
	>&(k-1)\ncolk(\C{n-1},k)\\>&3(k-1)\ncolk(\C{n-2},k-1)\\>& 3\ncolk(\C{n-1},k-1).
	\end{align*}
	Hence (b) is proved. We thus have
	\begin{align*}
	\ncolk(\C{n},4)=&3\ncolk(\C{n-1},4)+\ncolk(\C{n-1},3)\\
	>&25\ncolk(\C{n-1},3)\\>& \frac{25}{3}\ncolk(\C{n},3)\\>&8\ncolk(\C{n},3).
	\end{align*}
	which proves (c).
\end{proof}

\section{Upper bounds on $\avcol(G)$} \label{sec_ub}

We are now ready to give upper bounds on $\avcol(G)$. The following theorem gives a general upper bound on $\A(G)$ that is valid for all graphs $G$ of order $n$.

\begin{thm} \label{thm_ub}
Let  $G$ be a graph of order $n$, then,
$$
\avcol(G) \le n,
$$
with equality if and only if $G \simeq \K{n}$.
\end{thm}

\begin{proof}
	Clearly, $$\totcol(G)=\sum_{k=1}^nk\ncolk(G,k)\leq n\sum_{k=1}^n\ncolk(G,k)=n\ncol(G).$$ Hence, $\avcol(G)\leq n$, with equality if and only if $\ncolk(G,k)=0$ for all $k<n$, that is if $G\simeq \K{n}$. 
\end{proof}

 \noindent Since $\Delta(\K{n})=n-1$ we immediately get the following corollary to Theorem \ref{thm_ub}.

\begin{cor}
Let $G$ be a graph of order $n$ and maximum degree $\Delta(G)=n-1$. Then, $\avcol(G) \leq n,$ with equality if and only if $G \simeq \K{n}$.
\end{cor}

We now give a more precise upper bound on $\avcol(G)$ for graphs $G$ of order $n$ and maximum degree $\Delta(G)=n-2$.

\begin{thm}
Let $G$ be a graph of order $n \geq 2$ and maximum degree $\Delta(G)=n-2$. Then, $$\avcol(G) \leq \frac{n^2 - n + 1}{n},$$ with equality if and only if $G\simeq \K{n-1} \cup \K{1}$.
\end{thm}

\begin{proof}
Let $m$ be the number of edges in $G$, and let $x = \frac{n(n-1)}{2} - m = \ncolk(G,n-1)$. Then

\begin{align*}
\totcol(G) &= \displaystyle\sum_{k=1}^{n-2}k\ncolk(G,k) + x(n-1) + n \\&\leq (n-1) \displaystyle\sum_{k=1}^{n-2}\ncolk(G,k) + x(n-1) + n \\
		  &= \displaystyle(n-1)\sum_{k=1}^{n}\ncolk(G,k) + 1 \\&= (n-1)\ncol(G)+1. 
\end{align*}

Hence, $\avcol(G)\leq n-1+\frac{1}{\ncol(G)}$, with possible equality only if $\ncolk(G,k)=0$ for all $k<n-1$. It is proved in~\cite{Hertz16} that $\ncol(G) \geq n$, with equality if and only if $G$ is isomorphic to $\K{n-1} \cup \K{1}$ when $n\neq 4$, and $G$ is isomorphic to $\K{3} \cup \K{1}$ or $\C{4}$ when $n=4$. Since $\ncolk(\C{4},2)=1>0$ while $\ncolk(\K{n-1} \cup \K{1},k)=0$ for all $k<n-1$, we conclude that 
$\avcol(G) \leq n-1 + \frac{1}{n} = \frac{n^2 - n + 1}{n}$, with equality if and only if $G\simeq \K{n-1} \cup \K{1}$.

\end{proof}
\noindent The next simple case is when $\Delta(G)=1$.

\begin{thm} \label{thm:ub_delta1}
Let $G$ be a graph of order $n$ and maximum degree $\Delta(G) = 1$. Then,
$$
\A(G) \le \A(\floor*{\frac{n}{2}} \K{2}\cup (n\bmod 2)\K{1})$$
with equality if and only if $G\simeq \floor*{\frac{n}{2}} \K{2}\cup (n\bmod 2)\K{1}$.
\end{thm}

\begin{proof}
If $G$ contains two isolated vertices $u$ and $v$, we know from Proposition~\ref{prop:removeSimplicialEdge} that $\A(G+uv)>\A(G)$. Hence the maximum value of $\A(G)$ is reached when $G$ contains at most one isolated vertex,  that is $G\simeq \floor*{\frac{n}{2}} \K{2}\cup (n\bmod 2)\K{1}$. 
\end{proof}

We now give a precise upper bound on $\A(G)$ for graphs $G$ with maximum degree 2. We first analyze the impact of the replacement of an induced  $\C{n}$ ($n\geq 6$) by $\C{n-3}\cup \C{3}$.

\begin{lem} \label{lem:reduce_cycles}
$\avcol(G \cup \C{n}) < \avcol(G \cup \C{n-3} \cup \C{3})$ for all $n\geq 6$ and all graphs $G$.
\end{lem}

\begin{proof} We know from Lemma~\ref{lem:11inequalities} (a), (b) and (c) that the result is true for $n=6,7,8$. We can therefore assume $n\geq 9$.
	
	Let $f_n(k,k') = \ncolk(\C{n-3} \cup \C{3},k)\ncolk(\C{n},k')-\ncolk(\C{n},k)\ncolk(\C{n-3} \cup \C{3},k')$. Proposition~\ref{prop:crossProductAvCol} shows that it is sufficient to prove that $f_n(k,k')\geq 0$ for all  $k > k'$, the inequality being strict for at least one pair $(k,k')$.
	 Note that 
	 $f_n(n,2)=1>0$ for $n$ even. Also, $f_n(n,3)>0$ for $n$ odd. Indeed, this is true for $n=9$ since the values in Table \ref{tab_cross} give $f_n(9,3)=85-66=19$. For larger odd values of $n$, we proceed by induction, using Lemmas~\ref{prop:ncolkCn} and~\ref{lem:ncolkCUnion2}:
	 \begin{align*}
	 f_n(n,3)&=\ncolk(\C{n},3)-\ncolk(\C{n-3} \cup \C{3},3)\\
	 &=
	 \Big(2\ncolk(\C{n-1},3)+1\Big)-\Big(2\ncolk(\C{n-4} \cup \C{3},3)+6\Big)\\
	 &=\Big(4\ncolk(\C{n-2},3)+1\Big)-\Big(2\big(2\ncolk(\C{n-5} \cup \C{3},3)-6\big)+6\Big)\\
	 &=4\ncolk(\C{n-2},3)-4\ncolk(\C{n-5} \cup \C{3},3)+7\\
	 &=4f_{n-2}(n-2,3)+7>0.
	 \end{align*}

	Hence, it remains to prove that $f_n(k,k')\geq 0$ for all  $1\leq k' < k' \leq n$.	
	Let us start with the cases where $k'\leq 2$ and/or $k\geq n-1$.
\begin{itemize}
\item If $k'\leq 2$ then $f_n(k,k')=\ncolk(\C{n-3} \cup \C{3},k)\ncolk(\C{n},k')\geq 0$. 
\item
If $k\geq n-1$ then $\ncolk(\C{n},k) = \ncolk(\C{n-3} \cup \C{3},k)$ since \begin{itemize}
	\item $\ncolk(\C{n},n) = \ncolk(\C{n-3} \cup \C{3},n)=1$, and
	\item $\ncolk(\C{n},n-1) = \ncolk(\C{n-3} \cup \C{3},n-1)=\frac{n^2-3n}{2}$.
\end{itemize}
Also, it follows from Lemma~\ref{lem:ncolkCUnion} that $\ncolk(\C{n-3}\cup \C{3},k')=\ncolk(\Q{n},k')-(-1)^{n}\rho_{k'}$ and 
Equations \eqref{recs_minus} and \eqref{recs_plus} give
\begin{align*}
\ncolk(\C{n},k)&=\ncolk(\Path{n},k)-\ncolk(\C{n-1},k)\\
&=\left(\ncolk(\Q{n},k)+\ncolk(\Path{n-1},k)\right)-\left(\ncolk(\Path{n-1},k)-\ncolk(\C{n-2},k)\right)\\&=\ncolk(\Q{n},k)+\ncolk(\C{n-2},k).
\end{align*}
Altogether, this gives

\begin{align*}
f_n(k,k') &= \ncolk(\C{n},k)\Big(\ncolk(\C{n},k')-\ \ncolk(\C{n-3}\cup \C{3},k')\Big)\\
&=\ncolk(\C{n},k)\Big(\big(\ncolk(\Q{n},k')+\ncolk(\C{n-2},k')\big)-
\big(\ncolk(\Q{n},k')-(-1)^{n}\rho_{k'}\big)\Big)\\
&=\ncolk(\C{n},k)\Big(\ncolk(\C{n-2},k')+(-1)^{n}\rho_{k'}\Big).
\end{align*}

\vspace{-0.5cm}Hence,
\begin{itemize} 
	\item if $n$ is even and/or $k'\notin\{3,4\}$, then $f_n(k,k')\geq 0$;
	\item if $n$ is odd and $k'=3$ then 
	$f_n(k,k')=\ncolk(\C{n},k)(\ncolk(\C{n-2},3)-2)\geq 0$;
	\item
	If $n$ is odd and $k'=4$ then 
	$f_n(k,k')=\ncolk(\C{n},k)(\ncolk(\C{n-2},4)-1)\geq 0$.
\end{itemize}
\end{itemize}

\noindent We can therefore assume $3 \leq k' < k \leq n-2$ and we finally prove that 
$$f_n(k,k')\geq \begin{cases}
0&\mbox{if }k'\geq 6,\\
7\ncolk(\C{n},k)&\mbox{if } k'\in\{3,4,5\}.
\end{cases}$$
The values in the following table, computed with the help of those for $\C{9}$ and $\C{6}\cup \C{3}$ in Table \ref{tab_cross}, show that this is true for $n=9$:
\begin{center}
	\begin{tabular}{ c| c c c c c c c c c c}
		$(k,k')$ & (4,3) & (5,3) & (5,4) & (6,3) & (6,4) & (6,5)& (7,3) & (7,4) & (7,5) & (7,6)\\ \hline
		$f_9(k,k')$ & 8100 & 21973 & 55923 & 16366 & 53466 & 32046 & 4589& 16239 &12369 &2520 \\
		$7S(\C{9},k)$ & 5145 & 9849 & 9849 & 6468 & 6468 & 6468 & 1722& 1722& 1722& 1722
	\end{tabular}
\end{center}
For larger values of $n$, we proceed by induction. Lemmas  \ref{prop:ncolkCn} and \ref{lem:ncolkCUnion2} give
\begin{align}
f_n(k,k')=&\ncolk(\C{n},k')\ncolk(\C{n-3} \cup \C{3},k)-\ncolk(\C{n},k)\ncolk(\C{n-3} \cup \C{3},k') \nonumber\\
=&\ncolk(\C{n},k'){\Big(}(k{-}1)\ncolk(\C{n{-}4} {\cup} \C{3},k){+}\ncolk(\C{n{-}4} {\cup} \C{3},k{-}1){-}({-}1)^n\delta_k{\Big)} \nonumber\\ &{-}\ncolk(\C{n},k){\Big(}(k'{-}1)\ncolk(\C{n{-}4} {\cup} \C{3},k'){+}\ncolk(\C{n{-}4} {\cup} \C{3},k'{-}1){-}({-}1)^n\delta_{k'}{\Big)}\nonumber\\
=&{\Big(}(k'{-}1)\ncolk(\C{n{-}1},k'){+}\ncolk(\C{n{-}1},k'{-}1){\Big)}{\Big(}(k{-}1)\ncolk(\C{n{-}4} {\cup} \C{3},k){+}\ncolk(\C{n{-}4} {\cup} \C{3},k{-}1){\Big)} \nonumber\\
&{-}{\Big(}(k{-}1)\ncolk(\C{n{-}1},k){+}\ncolk(\C{n{-}1},k{-}1){\Big)}{\Big(}(k'{-}1)\ncolk(\C{n{-}4} {\cup} \C{3},k'){+}\ncolk(\C{n{-}4} {\cup} \C{3},k'{-}1){\Big)}\nonumber\\
&+(-1)^n\delta_{k'}\ncolk(\C{n},k)-(-1)^n\delta_k\ncolk(\C{n},k')\nonumber\\
=&(k-1)(k'-1)\Bigl(\ncolk(\C{n{-}1},k')\ncolk(\C{n{-}4} {\cup} \C{3},k)-\ncolk(\C{n{-}1},k)\ncolk(\C{n{-}4} {\cup} \C{3},k')\Bigr)\nonumber\\
&+(k'-1)\Bigl(\ncolk(\C{n{-}1},k')\ncolk(\C{n{-}4} {\cup} \C{3},k-1)-\ncolk(\C{n{-}1},k-1)\ncolk(\C{n{-}4} {\cup} \C{3},k')\Bigr)\nonumber\\
&+(k-1)\Bigl(\ncolk(\C{n{-}1},k'-1)\ncolk(\C{n{-}4} {\cup} \C{3},k)-\ncolk(\C{n{-}1},k)\ncolk(\C{n{-}4} {\cup} \C{3},k'-1)\Bigr)\nonumber\\
&+\ncolk(\C{n{-}1},k'-1)\ncolk(\C{n{-}4} {\cup} \C{3},k-1)-\ncolk(\C{n{-}1},k-1)\ncolk(\C{n{-}4} {\cup} \C{3},k'-1)\nonumber\\
&+(-1)^n\delta_{k'}\ncolk(\C{n},k)-(-1)^n\delta_k\ncolk(\C{n},k')\nonumber\\
=&(k-1)(k'-1)f_{n-1}(k,k')+(k'-1)f_{n-1}(k-1,k')\nonumber\\&+(k-1)f_{n-1}(k,k'-1)
 +f_{n-1}(k-1,k'-1)\nonumber\\&+(-1)^n\delta_{k'}\ncolk(\C{n},k)-(-1)^n\delta_k\ncolk(\C{n},k').\label{eq:17}
\end{align}

Since $\delta_{k}=0$ for $k\geq 6$ and $f_{n-1}(k,k')\geq 0$, $f_{n-1}(k-1,k')\geq 0$, $f_{n-1}(k,k'-1)\geq 0$, and $f_{n-1}(k-1,k'-1)\geq 0$ for $k>k'$, we have $f_n(k,k')\geq 0$ for $k>k'\geq 6$. 

Therefore, it remains to show that $f_n(k,k')\geq 7S(\C{n},k)$ for $k' \in \{3,4,5\}$. 
Let $g_n(k,k')=(-1)^n\delta_{k'}\ncolk(\C{n},k)-(-1)^n\delta_k\ncolk(\C{n},k')$. There are 4 possible cases.
\begin{itemize}
	\item {\it Case 1 : $k'\in \{4,5\}$ and $k\geq k'+2.$}\\
	We have $g_n(k,k')=(-1)^n\delta_{k'}\ncolk(\C{n},k)\geq -6\ncolk(\C{n},k)$. Using the induction hypothesis and Lemma \ref{prop:ncolkCn}, Equation \eqref{eq:17} gives
	\begin{align*}
f_n(k,k')
\geq & \Big(7(k-1)(k'-1)\ncolk(\C{n-1},k)+7(k'-1)\ncolk(\C{n-1},k-1)\Big)\\
&+\Big(7(k-1)\ncolk(\C{n-1},k)+7\ncolk(\C{n-1},k-1)\Big)-6\ncolk(\C{n},k)\\
=&7(k'-1)\ncolk(\C{n},k)+7\ncolk(\C{n},k)-6\ncolk(\C{n},k)
\\=&(7k'-6)\ncolk(\C{n},k)\\>&7\ncolk(\C{n},k).	\end{align*}
\item {\it Case 2 : $k'\in \{4,5\}$ and $k=k'+1.$}\\
Let us first give a lower bound on $g_n(k,k')$:
\begin{itemize}
	\item if $n$ is even and $k=6$, then $g_n(k,k')\geq 0$;
	\item if $n$ is even and $k=5$, then $g_n(k,k')\geq-\ncolk(\C{n},4)$, and we deduce from Lemma \ref{prop:ncolkCn2} (a) that $g_n(k,k')\geq -\ncolk(\C{n},5)$;
	\item if $n$ is odd, then $g_n(k,k')\geq-\delta_{k'}\ncolk(\C{n},k)\geq -6\ncolk(\C{n},k)$. 
\end{itemize}
Hence, whatever $n$ and $(k,k')$, $g_n(k,k')\geq -6\ncolk(\C{n},k)$. Since $f_{n-1}(k-1,k')=0$, using again the induction hypothesis and Lemma \ref{prop:ncolkCn}, we deduce from Equation \eqref{eq:17} that 
\begin{align*}
	f_n(k,k')
	\geq & \Big(7(k{-}1)(k'{-}1)\ncolk(\C{n{-}1},k)\Big)
	{+}\Big(7(k{-}1)\ncolk(\C{n{-}1},k){+}7\ncolk(\C{n{-}1},k{-}1)\Big){-}6\ncolk(\C{n},k)\\
	=&\Big(7(k'-1)\ncolk(\C{n},k)-7(k'-1)\ncolk(\C{n-1},k-1)\Big)+\Big(7\ncolk(\C{n},k)\Big)-6\ncolk(\C{n},k)\\
	=&(7k'-6)\ncolk(\C{n},k)-7(k'-1)\ncolk(\C{n-1},k-1)
	\end{align*}
	Since $k\leq 6$, Lemma \ref{prop:ncolkCn2} (b) shows that $\ncolk(\C{n-1},k-1)\leq \frac{1}{3}\ncolk(\C{n},k)$ and we therefore have 
	\begin{align*}f_n(k,k')&\geq \left(\frac{14k'-11}{3}\right)\ncolk(\C{n},k)\\&\geq 15\ncolk(\C{n},k)\\&>7\ncolk(\C{n},k).\end{align*}

 \item {\it Case 3 : $k'=3$ and $k\geq 5.$}\\
 As in the previous case, we have $g_n(k,k')\geq -6\ncolk(\C{n},k)$.
 The induction hypothesis gives 
 $f_{n-1}(k,k')\geq 7\ncolk(\C{n-1},k)$, $f_{n-1}(k-1,k')\geq 7\ncolk(\C{n-1},k-1)$, $f_{n-1}(k,k'-1){\geq} 0$, and  $f_{n-1}(k-1,k'-1)\geq 0$. Hence, Equation \eqref{eq:17} becomes 
 \begin{align*}	f_n(k,k')
 	\geq & 7(k-1)(k'-1)\ncolk(\C{n-1},k)+7(k'-1)\ncolk(\C{n-1},k-1)-6\ncolk(\C{n},k)\\
 	=&7(k'-1)\ncolk(\C{n},k)-6\ncolk(\C{n},k)\\
 	=&8\ncolk(\C{n},k)\\>&7\ncolk(\C{n},k).
 \end{align*}
 
  \item {\it Case 4 : $k'=3$ and $k=4.$}\\
  We have $g_n(k,k')=(-1)^n6\ncolk(\C{n},k)-(-1)^n6\ncolk(\C{n},k')$ and we know from  Lemma \ref{prop:ncolkCn2} (a) that $\ncolk(\C{n},4)>\ncolk(\C{n},3)$. Hence, $g_n(4,3)\geq -6(\ncolk(\C{n},k)-\ncolk(\C{n},k'))$. Using the induction hypothesis, Equation \eqref{eq:17} gives 
  \begin{align*} f_n(k,k')&\geq -7(k-1)(k'-1)\ncolk(\C{n-1},k)-6\Big(\ncolk(\C{n},k)-\ncolk(\C{n},k')\Big)\\
  &=42\ncolk(\C{n-1},k)-6\Big(\ncolk(\C{n},k)-\ncolk(\C{n},k')\Big).\end{align*}
  We therefore conclude from Lemmas \ref{prop:ncolkCn} and \ref{prop:ncolkCn2} (c) that 
  \begin{align*}f_n(4,3)&\geq \frac{42}{3}\Big(\ncolk(\C{n},4)-\ncolk(\C{n},3\Big)-6\Big(\ncolk(\C{n},4)-\ncolk(\C{n},3)\Big)\\&=8\Big(\ncolk(\C{n},4)-\ncolk(\C{n},3)\Big)\\
  &>8\Big(\ncolk(\C{n},4)-\frac{1}{8}\ncolk(\C{n},4)\Big)\\&=7\ncolk(\C{n},4).
 \end{align*}
\end{itemize}
\vspace{-0.5cm}\end{proof}

\noindent It is easy to check that 
\begin{itemize}
	\item $\A(\K{3}\cup \K{1})=\frac{13}{4}>3=\A(\C{4})$, and
	\item$\A(2\K{3}\cup \K{1})=\frac{778}{175}>\frac{684}{154}=\A(\K{3}\cup \C{4})$.
\end{itemize} Hence, $\A((p+1)\K{3}\cup \K{1})>\A(p\K{3}\cup \C{4})$ for $p=0,1$. We next prove that this inequality is reversed for larger values of $p$, that is $\A((p+1)\K{3}\cup \K{1})<\A(p\K{3}\cup \C{4})$ for $p\geq 2$. Proposition \ref{prop:crossProductAvCol} is of no help for this proof since, whatever $p$, there are pairs $(k,k')$ for which $\ncolk((p+1)\K{3}\cup \K{1},k) \ncolk(p\K{3}\cup \C{4},k') > \ncolk(p\K{3}\cup \C{4},k)\ncolk((p+1)\K{3}\cup \K{1},k')$, and other pairs for which the inequality is reversed. Also, it is not true that 
$$\A((p+1)\K{3}\cup \K{1})-\A(p\K{3}\cup \K{1})>\A(p\K{3}\cup \C{4})-\A((p-1)\K{3}\cup \C{4})$$
which would have given a simple proof by induction on $p$. The only way we have found to prove the desired result is to explicitly calculate $\A((p+1)\K{3}\cup \K{1})$ and $\A(p\K{3}\cup \C{4})$. This is what we do next, with the help of two lemmas.

\begin{lem} \label{lem_K_added}
If $G$ is a graph of order $n$, then
$$
\begin{array}{rcl}
\ncol(G \cup \K{2}) & = & \displaystyle\sum_{k=1}^n (k^2 + k + 1) \ncolk(G, k),\\[2ex]
\totcol(G \cup \K{2}) & = & \displaystyle\sum_{k=1}^n (k^3 + k^2 + 3k + 2) \ncolk(G, k),\\[2ex]
\ncol(G \cup \K{3}) & = & \displaystyle\sum_{k=1}^n (k^3 + 2k + 1) \ncolk(G, k),\\[2ex]
\totcol(G \cup \K{3}) & = & \displaystyle\sum_{k=1}^n (k^4 + 5 k^2 + 4k + 3) \ncolk(G, k).\\[2ex]
\end{array}
$$
\end{lem}

\begin{proof}
As observed in \cite{Hertz16},
\begin{equation} \label{eq_addingKr}
\ncolk(G\cup \K{r},k) = \sum_{i=0}^{r} {k-i \choose r-i} {r \choose i} (r-i)! \ \ncolk(G,k-i).
\end{equation}
For $r=2$, this gives $\ncolk(G\cup \K{2},k)=k(k-1) \ncolk(G,k) + 2(k-1) \ncolk(G,k-1) + \ncolk(G,k-2)$. Hence, 
\begin{align*}
\ncol(G \cup \K{2})=&\sum_{k=1}^{n+2}\ncolk(G\cup \K{2},k)\\=&\sum_{k=1}^{n+2}\Big(k(k-1) \ncolk(G,k) + 2(k-1) \ncolk(G,k-1) + \ncolk(G,k-2)\Big)\\
=&\sum_{k=1}^{n}k(k-1) \ncolk(G,k)+\sum_{k=1}^{n}2k \ncolk(G,k)+\sum_{k=1}^{n}\ncolk(G,k)\\=&\sum_{k=1}^{n}(k^2 + k + 1) \ncolk(G, k)
\end{align*}
and
\begin{align*}
\totcol(G \cup \K{2})=&\sum_{k=1}^{n+2} k\ncolk(G\cup \K{2}, k)\\=&\sum_{k=1}^{n+2}\Big(k^2(k-1) \ncolk(G,k) + 2k(k-1) \ncolk(G,k-1) + k\ncolk(G,k-2)\Big)\\
=&\sum_{k=1}^{n}k^2(k-1) \ncolk(G,k)+\sum_{k=1}^{n}2(k+1)k \ncolk(G,k)+\sum_{k=1}^{n}(k+2)\ncolk(G,k)\\
=&\sum_{k=1}^{n}(k^3 + k^2 + 3k + 2) \ncolk(G, k).
\end{align*}

\noindent The values for $\ncol(G \cup \K{3})$ and $\totcol(G \cup \K{3})$ are computed in a similar way.
\end{proof}
%
%
%

\begin{lem} \label{lem_K1K3_added}
If $G$ is a graph of order $n$, then,
$$
\begin{array}{rcl}
\ncol(G \cup \K{3} \cup \K{1}) & = &\displaystyle \sum_{k=1}^n (k^4 + k^3 + 5k^2 + 6k + 4) \ncolk(G, k),\\[2ex]
\totcol(G \cup \K{3} \cup \K{1}) & = & \displaystyle\sum_{k=1}^n (k^5 + k^4 + 9k^3 + 15k^2 + 21k + 13) \ncolk(G, k).
\end{array}
$$
\end{lem}

\begin{proof}
Let $G' = G \cup \K{3}$. Equation \eqref{eq_addingKr} gives
$\ncolk(G' \cup \K{1}, k) = k \ncolk(G',k) +  \ncolk(G',k-1)$. Hence, it follows from Lemma \ref{lem_K_added} that 

\begin{align*}
\ncol(G \cup \K{3} \cup \K{1})=&\sum_{k=1}^{n+4}\Big(k \ncolk(G',k) +  \ncolk(G',k-1)\Big)\\
=&\sum_{k=1}^{n+3}k\ncolk(G',k)+\sum_{k=1}^{n+3}\ncolk(G',k)\\
=&\totcol(G')+\ncol(G')\\
=&\sum_{k=1}^{n}\Bigl((k^4 + 5 k^2 + 4k + 3)+(k^3 + 2k + 1)\Bigr)\ncolk(G,k)\\
=&\sum_{k=1}^{n}(k^4 + k^3 + 5k^2 + 6k + 4)\ncolk(G,k).
\end{align*}
Equation \eqref{eq_addingKr} gives
\begin{align*}
&\sum_{k=1}^{n+3}k^2\ncolk(G',k)\\
=&\sum_{k=1}^{n}k^2\Big(
k(k-1)(k-2)\ncolk(G,k)\Big)+\sum_{k=1}^{n+1}k^2\Big(3(k-1)(k-2)\ncolk(G,k-1)\Big)\\&+\sum_{k=1}^{n+2}k^2\Big(3(k-2)\ncolk(G,k-2)\Big)+\sum_{k=1}^{n+3}k^2\ncolk(G,k-3)
\Big)\\
=&\sum_{k=1}^{n}k^3(k-1)(k-2)\ncolk(G,k)+
\sum_{k=1}^{n}3(k+1)^2k(k-1)\ncolk(G,k)\\
&+\sum_{k=1}^{n}3(k+2)^2k\ncolk(G,k)
+\sum_{k=1}^{n}(k+3)^2\ncolk(G,k)\\
=&\sum_{k=1}^{n}\Big(k^3(k-1)(k-2)+3(k+1)^2k(k-1)+3(k+2)^2k+(k+3)^2\Big)\ncolk(G,k)\\
=&\sum_{k=1}^{n}(k^5+8k^3+10k^2+15k+9)\ncolk(G,k).
\end{align*}
Hence, using again Lemma \ref{lem_K_added}, we get
\begin{align*}
\totcol(G \cup \K{3} \cup \K{1})=&\sum_{k=1}^{n+4}\Big(k^2 \ncolk(G',k) +k  \ncolk(G',k-1)\Big)\\
=&\sum_{k=1}^{n+3}k^2\ncolk(G',k)+
\sum_{k=1}^{n+3}(k+1)\ncolk(G',k)\\
=&\sum_{k=1}^{n+3}k^2\ncolk(G',k)+\totcol(G')+\ncol(G')\\
=&\sum_{k=1}^{n}\Bigl((k^5{+}8k^3{+}10k^2{+}15k{+}9){+}(k^4 {+} 5 k^2 {+} 4k {+} 3){+}(k^3 {+} 2k {+} 1)\Bigr)\ncolk(G,k)\\
=&\sum_{k=1}^{n}(k^5 + k^4 + 9k^3 + 15k^2 + 21k + 13)\ncolk(G,k).
\end{align*}
\end{proof}

%
%
%

 \noindent We are now ready to compare $\avcol(p \K{3}\cup \C{4})$ with $\avcol((p+1) \K{3}\cup \K{1})$.
\begin{thm} \label{lem_triangles}
\begin{itemize}
    \item[]\item[] 
$\avcol(p \K{3}\cup \C{4})<\avcol((p+1) \K{3}\cup \K{1})$ if $p=0,1$ and
\item[]$\avcol(p \K{3}\cup \C{4})>\avcol((p+1) \K{3}\cup \K{1})$ if $p\geq 2$.
\end{itemize}
\end{thm}

\begin{proof}
We have already mentioned that 
\begin{itemize}
	\item $\A(\K{3}\cup \K{1})=\frac{13}{4}>3=\A(\C{4})$, and
	\item$\A(2\K{3}\cup \K{1})=\frac{778}{175}>\frac{684}{154}=\A(\K{3}\cup \C{4})$.
\end{itemize}

Hence, it remains to prove that $\avcol(p \K{3}\cup \C{4})>\avcol((p+1) \K{3}\cup \K{1})$ for all $p\geq 2$. So assume $p\geq 2$ and let
$$f(p)= \totcol(p \K{3}\cup \C{4})\ncol((p+1) \K{3}\cup \K{1})-\ncol(p \K{3}\cup \C{4})\totcol((p+1) \K{3}\cup \K{1}).$$
Since $$\avcol(p \K{3}\cup \C{4}) - \avcol((p+1) \K{3}\cup \K{1}) = \frac{f(p)}{\ncol(p \K{3}\cup \C{4}))\ncol((p+1) \K{3}\cup \K{1})},
$$
we have to prove that $f(p)>0$.
Note that  Equations~\eqref{recs_minus} and~\eqref{recs_plus} give
\begin{align*}
\ncolk(G\cup \C{4},k)=&\ncolk(G\cup \Path{4},k)-\ncolk(G\cup \K{3},k)\\=&
\ncolk(G\cup \Q{4},k)+\ncolk(G\cup \Path{3},k)-\ncolk(G\cup \K{3},k)\\
=&\Big(\ncolk(G\cup \K{3}\cup \K{1},k)-\ncolk(G\cup \K{3},k)\Big)\\
&+
\Big(\ncolk(G\cup \K{3},k)+\ncolk(G\cup \K{2},k)\Big)-\ncolk(G\cup \K{3},k)\\
=&\ncolk(G\cup \K{3}\cup \K{1},k)-\ncolk(G\cup \K{3},k)+\ncolk(G\cup \K{2},k),
\end{align*}
which implies
\begin{align*}
\ncol(G\cup \C{4})=&\ncol(G\cup \K{3}\cup \K{1})-\ncol(G\cup \K{3})+\ncol(G\cup \K{2})\mbox{, and}\\
\totcol(G\cup \C{4})=&\totcol(G\cup \K{3}\cup \K{1})-\totcol(G\cup \K{3})+\totcol(G\cup \K{2}).
\end{align*}
Hence, with $G=p\K{3}$, we get
\begin{align*}
f(p) =& \totcol(G \cup \C{4})\ncol(G \cup \K{3} \cup \K{1}) - \totcol(G \cup \K{3} \cup \K{1})\ncol(G \cup \C{4})   \\
       =& \Big(\totcol(G \cup \K{3} \cup \K{1}) - \totcol(G \cup \K{3}) + \totcol(G \cup \K{2})\Big)\ncol(G \cup \K{3} \cup \K{1})  \\
       & - \totcol(G \cup \K{3} \cup \K{1})\Big(\ncol(G \cup \K{3} \cup \K{1}) - \ncol(G \cup \K{3}) + \ncol(G \cup \K{2})\Big) \\
       =& \ncol(G \cup \K{3} \cup \K{1})\Big(\totcol(G \cup \K{2}) - \totcol(G \cup \K{3})\Big) \\&- \totcol(G \cup \K{3} \cup \K{1})\Big(\ncol(G \cup \K{2}) - \ncol(G \cup \K{3})\Big). 
\end{align*}        
Since $\ncolk(G,k)=0$ for $k<3$, we deduce from Lemmas \ref{lem_K_added} and \ref{lem_K1K3_added} that 
\begin{align}
f(p) =& \sum_{k=1}^n a_k \ncolk(G, k) \sum_{k=1}^n b_k \ncolk(G,k) {-} \sum_{k=1}^n c_k \ncolk(G,k) \sum_{k=1}^n d_k \ncolk(G, k)\nonumber \\
=& \sum_{k=3}^n \sum_{k'=3}^n (a_k b_{k'} {-} c_k d_{k'}) \ncolk(G, k) \ncolk(G, k') \nonumber\\
=&\sum_{k=3}^n (a_k b_{k} {-} c_k d_{k}) \ncolk^2(G, k)\label{eq7}
\\&+\sum_{k'=3}^{n-1}\sum_{k=k'+1}^n(a_k b_{k'} - c_k d_{k'}+a_{k'} b_{k} - c_{k'} d_{k})\ncolk(G, k) \ncolk(G, k')\label{eq8}
\end{align}
where\begin{itemize}
	\item[] $\begin{array}{ll}a_k &= k^4 + k^3 + 5k^2 + 6k + 4\end{array}$, 
	\item[] $\begin{array}{ll}b_k& = (k^3 + k^2 + 3k + 2)-(k^4 + 5 k^2 + 4k + 3)\\&=-k^4 + k^3 -4k^2 - k - 1,\end{array}$ 
	\item[] $\begin{array}{ll}c_k& = k^5 + k^4 + 9k^3 + 15k^2 + 21k + 13\end{array}$, and 
	\item[] $\begin{array}{ll}d_k &= (k^2 + k + 1)-(k^3 + 2k + 1)\\&=-k^3+k^2-k.\end{array}$
\end{itemize}     

It is therefore sufficient to prove that the sums defined at 
\eqref{eq7} and \eqref{eq8}
are strictly positive.

\begin{itemize}
    \item Let $g(k) = a_k b_k - c_k d_k= k^6 + k^5 - 5k^4 - 19 k^3 - 19k^2 + 3k - 4$. It can be checked that $g(k)>0$ for all $k>3$. 
    Note that Equation \eqref{eq_addingKr} gives \begin{align*}
    \ncolk(G,3)=&\ncolk(p\K{3},3)=6\ncolk((p-1)\K{3},3)\\
    <&18\ncolk(p-1)\K{3},3)+24\ncolk((p-1)\K{3},4)
    \\=&\ncolk((p\K{3},4)=\ncolk(G,4).\end{align*}
    Since $g(3)=-112$ and $g(4) = 2328$, we have $g(3)\ncolk^2(G, 3)+g(4)\ncolk^2(G, 4)>0$, which implies
    $$\sum_{k=3}^n (a_k b_{k} {-} c_k d_{k}) \ncolk^2(G, k)=
    g(3)\ncolk^2(G, 3)+g(4)\ncolk^2(G, 4)+\sum_{k=5}^n g(k) \ncolk^2(G, k)>0.$$
    Hence, the sum in \eqref{eq7} is strictly positive.
    \item Let $h(k',k)=a_k b_{k'} - c_k d_{k'}+a_{k'} b_{k} - c_{k'} d_{k}.$ By definition of $a_k, b_k, c_k$ and $d_k$ we obtain
    \begin{align*}
h(k',k) =&
    \ {\left({k}^{3} - {k}^{2} + {k}\right)} {k'}^{5} \\&- {\left(2  {k}^{4} - {k}^{3} + 10  {k}^{2} + 6  {k} + 5\right)} {k'}^{4}\\
    & + {\left({k}^{5} + {k}^{4} + 20  {k}^{3} + 7  {k}^{2} + 35  {k} + 16\right)} {k'}^{3}  \\&- {\left({k}^{5} + 10  {k}^{4} - 7  {k}^{3} + 70  {k}^{2} + 35  {k} + 34\right)} {k}'^{2} \\
    & + {\left({k}^{5} - 6  {k}^{4} + 35  {k}^{3} - 35  {k}^{2} + 30  {k} + 3\right)} {k'} \\&- 5  {k}^{4} + 16 {k}^{3} - 34 {k}^{2} + 3  {k} - 8. 
    \end{align*}
 Let us make a change of variable. More precisely, we substitute $k'$ by $i+3$ and $k$ by $j+i+4$. Since $k'\geq 3$ and $k\geq k'+1$, we get $i\geq 0$ and $j\geq 0$. It is a tedious but easy exercise to check that with these new variables, $h(k',k)=h(i{+}3,j{+}i{+}4)=h'(i,j)$ with  
    \begin{align*}
h'(i, j) =& {\left(j^{2} + 2j + 3\right)} i^{6} + {\left(3j^{3} + 25j^{2} + 47j + 63\right)} i^{5}\\& + {\left(3j^{4} + 52j^{3} + 243j^{2} + 437j + 533\right)} i^{4}\\
    & + {\left(j^{5} + 37j^{4} + 338j^{3} + 1154j^{2} + 2017j + 2267\right)} i^{3}  \\
    & + {\left(8j^{5} + 161j^{4} + 997j^{3} + 2713j^{2} + 4692j + 4873\right)} i^{2} \\
    &+ {\left(22j^{5} + 290j^{4} + 1258j^{3} + 2729j^{2} + 4784j + 4443\right)} i  \\
    &+ 21j^{5} + 172j^{4} + 440j^{3} + 575j^{2} + 1112j + 602.
    \end{align*}
   Since $i\geq 0$, $j\geq 0$, and all coefficients in $h'(i, j)$ are positive, we conclude that $h'(i,j)=h(k',k) > 0$ for $3\leq k'<k\leq n$.
   
   Hence, the sum 
   $\displaystyle\sum_{k'=3}^{n-1}\sum_{k=k'+1}^nh(k',k)\ncolk(G, k) \ncolk(G, k')$
   in \eqref{eq8} is strictly positive.\qedhere
\end{itemize}
\vspace{-0.5cm}\end{proof}

We are now ready to prove the main result of this section. Let $\U{n}$ ($n \ge 3$) be the following graph.
$$
\U{n} = 
\begin{cases}
 \frac{n}{3} \K{3}  & \text{if } n\ \mymod 3 = 0, \text{ and } n \ge 3,\\
 \frac{n-1}{3}\K{3} \cup \K{1} & \text{if } n = 4 \text{ or } n = 7,\\
\frac{n-4}{3} \K{3} \cup \C{4} & \text{if }  n\ \mymod 3 = 1,\text{ and } n \ge 10,\\
 \frac{n-5}{3} \K{3} \cup \C{5} & \text{if } n\ \mymod 3 = 2,\text{ and } n \ge 5.
\end{cases}
$$

\begin{thm} \label{thm_ub_delta2}
If $G$ is a graph of order $n \geq 3$ and maximum degree $\Delta(G)=2$, then, $$\avcol(G) \leq \avcol(\U{n}),$$ with equality if and only if $G \simeq \U{n}$.
\end{thm}

\begin{proof}
Since $\Delta(G) = 2$, $G$ is a disjoint union of cycles and paths. Now, suppose that $G$ maximizes $\A$ among all graphs of maximum degree 2. Then at most one connected component of $G$ is a path. Indeed, if $G\simeq G'\cup \Path{k}\cup \Path{k'}$, then Equations~\eqref{rec_minus} and~\eqref{rec_plus} give $\ncol(G'\cup \Path{k}\cup \Path{k'})=\ncol(G'\cup \Path{k+k'})+\ncol(G'\cup \Path{k+k'-1})$ and 
$\totcol(G'\cup \Path{k}\cup \Path{k'})=\totcol(G'\cup \Path{k+k'})+\totcol(G'\cup \Path{k+k'-1})$. Moreover, we know from Proposition~\ref{prop:removeSimplicialEdge} that $\avcol(G'\cup \Path{k+k'-1})<\avcol(G'\cup \Path{k+k'})$. Hence, Proposition~\ref{prop:pro_avcolInfSum} implies that $\avcol(G)=\avcol(G'\cup \Path{k}\cup \Path{k'})<\avcol(G'\cup \Path{k+k'})$. Since $(G'\cup \Path{k+k'})$ is of order $n$ and maximum degree 2, this contradicts the hypothesis that $G$ maximizes $\A$.

We know from Lemma~\ref{lem:12} that replacing a path $\Path{k}$ of order $k \ge 3$ by a cycle $\C{k}$ strictly increases $\A(G)$. Moreover, Lemma~\ref{lem:reduce_cycles} shows that replacing a cycle $\C{k}$ of order $k \ge 6$ by $\C{k-3} \cup \K{3}$ increases $\avcol(G)$. Hence $G$ is a disjoint union of copies of $\K{3}$, $\C{4}$ and $\C{5}$ and eventually one path that is either $\K{1}$ or $\K{2}$. 

Considering Lemma~\ref{lem:11inequalities}, we know from (d), (e) and (f) that $G$ does not contain $\K{2}$, and from (g)-(k)  that at most one connected component of $G$ is not a $\K{3}$. Hence, if $n \bmod 3=0$ then $G\simeq \frac{n}{3} \K{3}$ and if $n \bmod 3=2$ then $G\simeq \frac{n-5}{3} \K{3} \cup \C{5}$. Finally, Theorem~\ref{lem_triangles} shows that $G\simeq 
\frac{n-1}{3}\K{3} \cup \K{1}$ if $n = 4$ or $7$, and $G\simeq 
\frac{n-4}{3} \K{3} \cup \C{4}$ if $n \bmod 3 = 1$ and $n \ge 10$.

\end{proof}

\section{Concluding remarks}

We have given a general upper bound on $\avcol(G)$ that is valid for all graphs $G$, and a more precise one for graphs of order $n$ and maximum degree $\Delta(G)\in \{1,2,n-2\}$. Note that there is no known lower bound on $\avcol(G)$ which is a function of $n$ and such that there exists at least one graph of order $n$ which reaches it.

The problem of finding a tight upper bound for graphs with maximum degree in $\{3,\ldots,n-3\}$ remains open. Since all graphs of order $n$ and maximum degree $\Delta(G)\in\{1,n-2,n-1\}$ that maximize $\avcol(G)$ are isomorphic to $\floor*{\tfrac{n}{\Delta(G)+1}} \K{\Delta(G)+1}\cup \K{n\bmod (\Delta(G)+1)}$ (but this is not always true for $\Delta(G)=2$), one could be tempted to think that this is also true when $3\leq \Delta(G)\leq n-3$. We have checked this statement by enumerating all graphs having up to 11 vertices, using \PHOEG~\cite{PHOEG}. We have thus determined that there is only one graph of order $n\leq 11$ and $\Delta(G)\neq 2$ (among more than a billion), namely $\CBar{6} \cup \K{4}$, for which such a statement is wrong. Indeed, $\A(\CBar{6} \cup \K{4}) = 5.979 > 5.967=\A(2 \K{4} \cup \K{2}) $, which shows that $2 \K{4} \cup \K{2}$ does not maximize $\A(G)$ among all graphs of order 10 and maximum degree 3.

\section*{Acknowledgments}

The authors thank Julien Poulain for his precious help in optimizing our programs allowing us to check our conjectures on a large number of graphs.

Computational resources have been provided by the Consortium des \'{E}quipements de Calcul Intensif (C\'{E}CI), funded by the Fonds de la Recherche Scientifique de Belgique (F.R.S.-FNRS) under Grant No. 2.5020.11 and by the Walloon Region.

\bibliographystyle{acm}
\bibliography{avcolUB}

\begin{thebibliography}{10}

\bibitem{absil}
{\sc Absil, R., Camby, E., Hertz, A., and M\'elot, H.}
\newblock A sharp lower bound on the number of non-equivalent colorings of
  graphs of order $n$ and maximum degree $n-3$.
\newblock {\em Discrete Appl. Math. 234\/} (2018), 3--11.
\newblock Special Issue on the Ninth International Colloquium on Graphs and
  Optimization (GO IX), 2014.

\bibitem{PHOEG}
{\sc Devillez, G., Hauweele, P., and M{\'e}lot, H.}
\newblock {PHOEG Helps to Obtain Extremal Graphs}.
\newblock In {\em Operations Research Proceedings 2018 (GOR (Gesellschaft fuer
  Operations Research e.V.))\/} (sept. 12-14 2019), B.~Fortz and M.~Labb\'e,
  Eds., Springer, Cham, p.~251 (Paper 32).

\bibitem{Diestel00}
{\sc Diestel, R.}
\newblock {\em {Graph Theory}}, second edition~ed.
\newblock Springer-Verlag, 2017.

\bibitem{DKT05}
{\sc Dong, {\protect F. M}., Koh, {\protect K. M}., and Teo, {\protect K. L}.}
\newblock {\em {Chromatic polynomials and chromaticity of graphs}}.
\newblock World Scientific Publishing Company, 2005.

\bibitem{Duncan10}
{\sc Duncan, B.}
\newblock Bell and {S}tirling numbers for disjoint unions of graphs.
\newblock {\em Congressus Numerantium 206\/} (01 2010).

\bibitem{DP09}
{\sc Duncan, B., and Peele, {\protect R. B}.}
\newblock Bell and {S}tirling numbers for graphs.
\newblock {\em J. Integer Seq. 12\/} (2009).
\newblock Article 09.7.1.

\bibitem{GT13}
{\sc Galvin, D.and~Thanh, {\protect D.T}.}
\newblock Stirling numbers of forests and cycles.
\newblock {\em Electron. J. Comb. 20\/} (2013).
\newblock Paper P73.

\bibitem{Hertz21}
{\sc Hertz, A., Hertz, A., and M\'elot, H.}
\newblock {Using graph theory to derive inequalities for the Bell numbers}.
\newblock Submitted. arXiv:2104.00552, 2021.

\bibitem{Hertz16}
{\sc Hertz, A., and M\'elot, H.}
\newblock Counting the number of non-equivalent vertex colorings of a graph.
\newblock {\em Discrete Appl. Math. 203\/} (2016), 62--71.

\bibitem{HertzLB}
{\sc Hertz, A., M\'elot, H., Bonte, S., and Devillez, G.}
\newblock Lower bounds and properties for the average number of colors in the
  non-equivalent colorings of a graph.
\newblock Submitted. arXiv:2104.14172, 2021.

\bibitem{KN14}
{\sc Keresk\'enyi-Balogh, Z., and Nyul, G.}
\newblock Stirling numbers of the second kind and {B}ell numbers for graphs.
\newblock {\em Australas. J. Comb. 58\/} (2014), 264--274.

\bibitem{OR}
{\sc Odlyzko, A., and Richmond, L.}
\newblock On the number of distinct block sizes in partitions of a set.
\newblock {\em J. Comb. Theory Ser. A. 38}, 2 (1985), 170--181.

\bibitem{Sloane}
{\sc Sloane, N.}
\newblock The on-line encyclopedia of integer sequences.
\newblock http://oeis.org.

\end{thebibliography}

\end{document}